\documentclass[11pt]{amsart}

\usepackage{amssymb}
\usepackage{amsmath,amsthm,geometry,enumerate}
\usepackage[dvips]{graphicx}
\usepackage{psfrag}
\usepackage{amscd}
\usepackage{xspace}
\usepackage[all]{xy}

\normalfont
\newcommand{\RedefinitSymbole}[1]{%
  \expandafter\let\csname old\string#1\endcsname=#1 \let#1=\relax
\newcommand{#1}{\csname old\string#1\endcsname\,}%
}
\RedefinitSymbole{\forall}
\RedefinitSymbole{\exists}

\def\cf{\textit{cf.}\kern.3em}

\def\resp{\textit{resp.}\kern.3em}
\renewcommand{\k}{\kern2pt}

\numberwithin{equation}{section} 

\DeclareMathOperator{\Aut}{Aut}

\DeclareMathOperator{\spec}{Spec}

\DeclareMathOperator{\ord}{ord}

\newtheorem{proposition}[equation]{Proposition}
\newtheorem{theorem}[equation]{Theorem}
\newtheorem{corollary}[equation]{Corollary}

\theoremstyle{definition}
\newtheorem{definition}[equation]{Definition}

\newtheorem{remark}[equation]{\textbf{Remark}}

\newcommand{\mun}{$\mathcal{M}_{1,n}$}
\newcommand{\mbun}{$\overline{\mathcal{M}}_{1,n}$}

\newcommand{\virg}[1]{\textquotedblleft#1\textquotedblright}

\begin{document}

\date{}
\title{{The orbifold cohomology of moduli of hyperelliptic curves}}
\author{Nicola Pagani}
\address{\flushleft KTH Matematik, Lindstedtsv\"agen 25, S-10044 Stockholm\\}
\email{pagani@kth.se}

\begin{abstract}We study the inertia stack of $[\mathcal{M}_{0,n}/S_n]$, the quotient stack of the moduli space of smooth genus $0$ curves with $n$ marked points via the action of the symmetric group $S_n$. Then we see how from this analysis we can obtain a description of the inertia stack of $\mathcal{H}_g$, the moduli stack of hyperelliptic curves of genus $g$. From this, we can compute additively the Chen--Ruan (or orbifold) cohomology of $\mathcal{H}_g$.\end{abstract}
\maketitle 
\section{Introduction}

A hyperelliptic curve of genus $g$ is a smooth algebraic curve that admits a $2:1$ map to $\mathbb{P}^1$, and thus has $2g+2$ branch points. From its very definition, it is clear that the moduli stack of genus $g$ hyperelliptic curves $\mathcal{H}_g$ admits a map onto the moduli stack $[\mathcal{M}_{0,2g+2}/S_{2g+2}]$, which is an isomorphism at the level of coarse moduli spaces. The foundations for moduli of hyperelliptic curves, as well as the precise definition of the previous map, can be found in \cite{lonsted} (in particular Theorem 5.5).

The last decade has seen tremendous improvements in our understanding of the moduli space of hyperelliptic curves $\mathcal{H}_g$. We mention here some of the recent achievements that are relevant to the present work. In the paper \cite{arsievistoli}, $\mathcal{H}_g$ is described as a moduli stack of cyclic covers of the projective line. As a consequence of this description, the authors are able to determine its Picard group. Along these lines, the Picard group of the Deligne-Mumford compactification $\overline{\mathcal{H}}_g$ was computed (see \cite{cornpic}), and very recently the whole integral Chow ring of  $\mathcal{H}_g$ was computed in \cite{fulghesu} (see also \cite{edidin}, \cite{gorviv}). In the last years, much effort was also made in studying the automorphism groups of hyperelliptic curves \cite{shaska1}, \cite{shaska2}, \cite{shaska3}, \cite{shaska4}.

In this paper we deal with rational cohomology and Chow group with rational coefficients. From both these points of view, the moduli stacks $\mathcal{H}_g$ are trivial. The triviality of $H^*(\mathcal{H}_g, \mathbb{Q})$ follows from \cite[Theorem 2.13]{kisinlehrer}, while the triviality of $A^*_{\mathbb{Q}}(\mathcal{H}_g)$ follows from its description as finite quotient of the affine variety $\mathcal{M}_{0,n}$. Still some nontriviality can be measured with rational coefficients, but one has to consider instead the \emph{orbifold cohomology} or the \emph{stringy Chow group}. The orbifold cohomology as a vector space (or Chen--Ruan cohomology) of an orbifold $\mathcal{X}$ is obtained by adding to the usual cycles of $\mathcal{X}$ the cycles of all the \emph{twisted sectors} of $\mathcal{X}$. The twisted sectors are orbifolds that parametrize pairs $(x, g)$ where $x$ is a point of $\mathcal{X}$ and $g \in \Aut(x)$. The new cycles are then given an unconventional degree, which is the sum of their original degree as cycles inside their twisted sector $Y$, plus a rational number (called \emph{age} or \emph{degree shifting number}) that depends on the normal bundle $N_Y \mathcal{X}$. 

The orbifold cohomology of moduli spaces of curves is studied in \cite{pagani1}, \cite{pagani2}, \cite{spencer2} (see also the PhD thesis \cite{paganitesi}, \cite{spencer}). The present work has some nontrivial intersection with \cite{pagani2} and \cite{spencer2}, since in these two papers in particular the orbifold cohomology and stringy Chow group of $\mathcal{M}_2= \mathcal{H}_2$ are described.

The main result of this paper is Theorem \ref{principale}, where we give for any $g$ a closed formula for the \emph{orbifold Poincar\'e polynomial} of $\mathcal{H}_g$, that is, a \virg{polynomial}\footnote{We call it polynomial in analogy with the ordinary Poincar\'e polynomial, although the exponents of the variable $q$ are not natural but rational.}, whose coefficient of $q^i$ corresponds to the dimension of the group $H^{i}$.

To achieve this result, we first describe in Section \ref{sezione2} the twisted sectors of $[\mathcal{M}_{0,n}/S_n]$ as quotients of certain $\mathcal{M}_{0,k}$ modulo a subgroup of $S_k$. 

Then, in Section \ref{sezione3}, we study the twisted sectors of $\mathcal{H}_g$. If $g$ is odd, we see that the twisted sectors of $\mathcal{H}_g$ are simply the twisted sectors of $[\mathcal{M}_{0,2g+2}/S_{2g+2}]$ repeated twice. If $g$ is even, most of the twisted sectors of $\mathcal{H}_g$ correspond to the twisted sectors of $[\mathcal{M}_{0,2g+2}/S_{2g+2}]$, whose distinguished automorphism is not an involution, repeated twice. The remaining few twisted sectors of $\mathcal{H}_g$ are still described as quotients of moduli of genus $0$, pointed curves modulo the action of a certain subgroup of the symmetric group on the marked points.

Finally, in Section \ref{sezione4} we compute all the degree shifting numbers, and we write the explicit results by recollecting the results of the previous sections.

\subsection{Notation} We work over $\mathbb{C}$; cohomologies and Chow groups are taken with rational coefficients. Orbifold for us means smooth Deligne--Mumford stack, and we always work within the category of Deligne--Mumford stacks. If a finite group $G$ acts on a scheme (stack) $X$, $[X/G]$ is the stack quotient and $X/G$ is the quotient as a scheme. We call $\mu_N:= \mathbb{Z}_N^{\vee}$ the group of characters of $\mathbb{Z}_N$, and $\mu_N^*$ the subgroup whose elements are the invertible characters. We make an implicit use of the relative language of schemes. For instance, when no confusion can arise, we speak of a genus $g$ smooth curve, meaning a family of genus $g$ smooth curve over a certain base $S$.

\section{Definition of Orbifold Cohomology}
In this section we define orbifold cohomology. For a more detailed study of this topic, we address the reader to \cite[Section 3]{agv2} for the various inertia stacks, and to \cite[Section 7.1]{agv2} for the degree shifting number (the original reference is \cite{chenruan}). What we call orbifold cohomology is the graded vector space underlying the Chen--Ruan cohomology ring (or algebra): the latter is a more refined object that we will not introduce in this work.

We introduce the following natural stack associated to a Deligne--Mumford stack $X$, which points to where $X$ fails to be an algebraic space.

\begin{definition} \label{definertia} (\cite[4.4]{agv1}, \cite[Definition 3.1.1]{agv2}) Let $X$ be an algebraic stack. The \emph{inertia stack} $I(X)$ of $X$ is defined as
$$
I(X) := \coprod_{N \in \mathbb{N}} I_N(X)
$$
where $I_N(X)(S)$ is the following groupoid:
\begin{enumerate}
\item The objects are pairs $(\xi, \alpha)$, where $\xi$ is an object of $X$ over $S$, and $\alpha\colon \mu_N \to \Aut(\xi)$ is an injective homomorphism.
\item The morphisms are the morphisms $g\colon \xi \to \xi'$ of the groupoid $X(S)$, such that $g \cdot \alpha(1)= \alpha'(1) \cdot g$.  
\end{enumerate}
We also define $I_{TW}(X):= \coprod_{N > 1}I_N(X)$, in such a way that $$I(X)=I_1(X) \coprod I_{TW}(X).$$ The connected components of $I_{TW}(X)$ are called \emph{twisted sectors} of the inertia stack of $X$, or also twisted sectors of $X$. The inertia stack comes with a natural forgetful map $f\colon I(X) \to X$.

\end {definition}

We observe that, by our very definition, $I_N(X)$ is an open and closed substack of $I(X)$, but it rarely happens  that it is connected. One special case is when $N$ equals to $1$: in this case the map $f$ restricted to $I_1(X)$ induces an isomorphism of the latter with $X$. The connected component $I_1(X)$ will be referred to as the \emph{untwisted sector}.

 We also observe that given a generator of $\mu_N$, we obtain an isomorphism of $I(X)$ with $I'(X)$, where the latter is defined as the ($2$-)fiber product $X \times_{X \times X} X$ where both morphisms $X \rightarrow X \times X$ are the diagonals. Over more general fields than $\mathbb{C}$, the object we defined in \ref{definertia} is usually called \emph{cyclotomic inertia stack}, whereas the \emph{inertia stack} is the above ($2$-)fiber product. We address the reader to \cite[Section 3]{agv2} for more details on the inertia stack and its variants.

\begin{remark} \label{mappaiota} There is an involution $\iota\colon I_N(X) \to I_N(X)$, which is induced by the map $\iota'\colon \mu_N \to \mu_N$, that is $\iota'(\zeta):= \zeta^{-1}$.
\end{remark}

\begin{proposition} \label{liscezza1} \cite[Corollary 3.1.4]{agv2} Let $X$ be a smooth algebraic stack. Then the stacks $I_N(X)$ (and therefore $I(X)$ itself) are smooth.
\end{proposition}

We now study the behaviour of the inertia stack under arbitrary morphisms of stacks. 

\begin {definition}\label{pullinertia} Let $f\colon X \to Y$ be a morphism of stacks. We define $f^*(I(Y))$ as the stack that makes the following diagram $2-$cartesian: $$
\xymatrix{f^*(I(Y)) \ar[r]^{I(f)} \ar[d] \ar@{}|{\square}[dr] & I(Y) \ar[d] \\
X \ar[r]^{f} & Y 
}$$
and $I(f)$ as the map that lifts $f$ in the diagram. Obviously, there is an induced map that we call $I'(f)$, which maps $I(X) \to f^*(I(Y))$.
\end {definition}

We now define the degree shifting number for the twisted sectors of the inertia stack of a smooth stack $X$. With $R{\mu_N}$, we denote the representation ring of $\mu_N$.

\begin {definition} \cite[Section 7.1]{agv2} Let $\rho\colon\mu_N \to \mathbb{C}^*$ be a group homomorphism. It is determined by an integer $0 \leq k \leq N-1$ as $\rho( \zeta_N)= \zeta_N^k$. We define a function \emph{age}:
$$
\textrm{age}(\rho)=k/N.
$$
This function extends to a unique group homomorphism:
$$
\textrm{age}\colon R \mu_N \to \mathbb{Q}.
$$
\end {definition}

\noindent We now define the age of a twisted sector $Y$ of a smooth stack $X$. Let $f$ be the restriction to $Y$ of the natural forgetful map $I(X) \to X$.
\begin{definition} \label{definitionage} (\cite[Section 3.2]{chenruan}, \cite[Definition 7.1.1]{agv2}) Let $Y$ be a twisted sector and $g\colon \spec \mathbb{C} \to Y$ a point. 
Then the pull-back via $f \circ g$ of the tangent sheaf, $(f \circ g)^*(T_X)$, is a representation of $\mu_N$ on a finite dimensional vector space. 
 We define $$a(Y):= \textrm{age}((f \circ g)^*(T_X)).$$
\end{definition}

\noindent We can then define the orbifold, or Chen--Ruan, degree.

\begin{definition} \label{defcoomorb2} (\cite[Definition 3.2.2]{chenruan}) We define the $d-th$ degree orbifold cohomology group as follows:
$$
H^d_{CR}(X, \mathbb{Q}):= \bigoplus_i  H^{d-2 a(X_i,g_i)}(X_i, \mathbb{Q})
$$
where the sum is over all twisted sectors. The \emph{orbifold Poincar\'e polynomial} of $X$ is
$$
P^{CR}_{X}(q):= \sum_{i \in \mathbb{Q}^+} \dim \left(H^{i}_{CR}(X) \right) q^i .
$$
\end{definition}
\begin{remark} \label{stringychow} One can also define the stringy Chow group and its unconventional grading in complete analogy with the above definition. See \cite[Section 7.3]{agv2} for this construction.
\end{remark}

\section{The inertia stack of the configuration of unordered points on the Riemann sphere} \label{sezione2}

In this section, we study the cohomology of the inertia stack of $[\mathcal{M}_{0,n}/S_n]$ (also known in the literature as $\widetilde{\mathcal{M}}_{0,n}$). For this, it is enough to give a description of the coarse moduli spaces of the twisted sectors of the inertia stack of $[\mathcal{M}_{0,n}/S_n]$. We thus describe the coarse moduli spaces of the twisted sectors of the latter stack as quotients of the kind $\mathcal{M}_{0,k}/S$, for $S$ a certain subgroup of $S_k$. The cohomology of these quotients is well known. The cohomology of $\mathcal{M}_{0,n}$ was first computed as a representation of the symmetric group $S_n$ by Getzler \cite[5.6]{getzleroperads} (see also \cite{kisinlehrer}).

In particular, we shall use the following result.
\begin{proposition} The Poincar\'e polynomial of $\mathcal{M}_{0,n+2}/S_n$ is
$$
P^0_{n+2;n,1,1}(q)=\sum_{i=0}^{n-1} q^i .
$$
The Poincar\'e polynomial of $\mathcal{M}_{0,n+2}/S_n \times S_2$ is
$$
P^0_{n+2;n,2}(q) =\begin{cases} 1 & n=1 \\ \sum_{i=0}^{\lfloor \frac{n-2}{4}\rfloor} q^i+ q^{i+1} & n>1 .
\end{cases}
$$
\end{proposition}
\begin{proof} It follows from \cite[Theorem 2.9]{kisinlehrer}.
\end{proof}

It will be convenient to have a definition for the set where each invertible character of $\mathbb{Z}_N$ is identified with its inverse:
\begin{definition} \label{identificati} We define $\widetilde{\mu}_N^*$ as the quotient set $\mu_N^*/\mathbb{Z}_2$ where $\overline{1} (\zeta_N):= \zeta_N^{-1}$. If $N$ is even, we define $\overline{\mu}_N^*$ to be the quotient set $\mu_N^*/\mathbb{Z}_2$ where the action of $\overline{1}$ is defined to be: $\overline{1} (\zeta_N):= -\zeta_N^{-1}$.
\end{definition}
The following proposition describes the inertia stack of $[\mathcal{M}_{0,n}/S_n]$. 

\begin{proposition} \label{inertian} We describe the coarse moduli spaces of the twisted sectors of $[\mathcal{M}_{0,n}/S_n], n \geq 3$.
\begin{enumerate}
\item Suppose $N>2$, or $n$ odd. If there exists $a \in \{0,1,2\}$ such that $n=kN+a$: $$I_N([\mathcal{M}_{0,n}/S_n])= \begin{cases} \coprod_{\chi \in \widetilde{\mu}_N^*} \left(\mathcal{M}_{0,k+2}/S_k, \chi \right) & a=0,2\\\coprod_{\chi \in {\mu_N^*}} \left(\mathcal{M}_{0,k+2}/S_k , \chi \right) & a=1 \end{cases}$$
and $I_N([\mathcal{M}_{0,n}/S_n])$ is empty otherwise.
\item if $n$ is even, $n=:2g+2$: $$I_2([\mathcal{M}_{0,n}/S_n]) = \left(\mathcal{M}_{0, g+2}/S_g \times S_2,-1 \right) \coprod \left(\mathcal{M}_{0, g+3}/S_{g+1} \times S_2,-1 \right) $$
\end{enumerate}

\begin{proof}

Let $C$ be a smooth genus $0$ curve and $\alpha$ an automorphism of finite order $N$ of it. From the Riemann--Hurwiz formula, $\alpha$ has at least two fixed points, and if it had more it would be the identity. We can choose coordinates on $C$ in such a way that the two fixed points are $0$ and $\infty$, and $\alpha$ is the multiplication by a primitive $N-$th root of unity. Now let $C'= C / \langle \alpha \rangle$ be the quotient curve. Given a suitable choice of coordinates on $C$ and $C'$, the quotient map $C \to C'$ becomes the map $z \to z^N$. An automorphism $\alpha$ of $C$ is an automorphism of $C$ with  $n$ unordered points if the unordered points are chosen in the subsets of $C$ that are invariant under the action of $\langle \alpha \rangle$. The finite subsets of $C$ invariant
under $\langle \alpha \rangle$ contain exactly $k N + 2$ points: $k$ orbits of $N$ elements each, plus the two points fixed by $\alpha$.

Let us first deal with the case when $N>2$. In this case there is at most one choice of $k \in \mathbb{N}^+$ and $a \in \{0,1,2\}$ such that $n=kN +a$. The set of $n=k N+a$ marked points corresponds to a set of $k+a$ marked points on $C'$, where the $a$ points are a subset of the branch divisor. 

If $a$ is equal to $1$, there is a choice of a point $p$ in $C$ that is the only point that is both in the set of $n$ points, and a fixed point for $\alpha$. Then $\alpha$ determines a $\chi \in \mu_N^*$: the character of the representation of $\alpha$ on $T_p C$. Conversely, from the set of $k$ unordered points on $C'$ plus two branch points and a character $\chi \in \mu_N^*$, one can reconstruct the set of $n=k N+1$ points on $C$ and the automorphism $\alpha$.

If $a$ equals $0$ or $2$, there is no such a choice of a distinguished point $p$ in the set of two fixed points of $\alpha$. Then $\alpha$ acts on the set of fixed points, thus determining two inverse characters in $\mu_N^*$, and an ordering of the two ramification, equivalently branch, points. In the same way as before, $\alpha$ gives then an equivalence class in the set $\widetilde{\mu}_N^*$ (see Definition \ref{identificati}). Conversely, from the set of $k$ unordered points on $C'$ plus two ordered branch points and an element $\chi \in \widetilde{ \mu}_N^*$, one can reconstruct the set of $n=k N+a$ points on $C$ and the automorphism $\alpha$.

If $N$ is equal to $2$, the argument is the same, the only differences being that $\widetilde{\mu}_2^*= \mu_2^*$ and that, being $\alpha$ an involution, it does not distinguish between the two ramification (equivalently, branch) points of $C \to C'=C/ \langle \alpha \rangle$.
\end{proof}
\end{proposition}

\begin{remark} \label{stackyremark} In \cite{agv2}, the authors introduce two notions related to the inertia stack: the \emph{stack of cyclotomic gerbes} (\cite[Definition 3.3.6]{agv2}) and the \emph{rigidified inertia stack} (\cite[3.4]{agv2}), showing in \cite[3.4.1]{agv2} that they are equivalent.  By substituting the quotients $\mathcal{M}_{0,k+2}/S_k$ with the \emph{stack} quotients $[\mathcal{M}_{0,k+2}/S_k]$ one obtains a stacky description of the rigidified inertia stack of $[\mathcal{M}_{0,n}/S_n]$. We have stated the earlier-mentioned proposition in this simplified way because this is enough for our purposes, and in this way we could avoid having to introduce the whole theory of inertia stack and its variants (see \cite[Section 3]{agv2}).
\end{remark}


We propose here a non-trivial check of the correctness of the delicate, albeit relatively elementary, result of Proposition \ref{inertian}. For a stack $X$ we have two notions of Euler characteristics: the topological Euler characteristic $\chi(X)$ of the associated coarse moduli space, and the \emph{orbifold} Euler characteristic $e(X)$. The latter is not necessarily an integer, and when $G$ is a finite group acting on a scheme $X$, such that $X=[Y/G]$, it satisfies $|G| \cdot e(X)= \chi(Y)$. It is well-known that these two quantities are related by the equality \begin{equation}\label{equat} \chi(X)=e(I(X)),\end{equation} see \cite[p.21]{behrend}. 
The latter equality provides a consistency check of the corrected version of Proposition \ref{inertian}, together with Remark \ref{stackyremark}. Let us fix $X= [\mathcal{M}_{0,n}/S_n]$: after multiplying by $\frac{1}{N}$ the orbifold Euler characteristic of each rigidified $I_N$, the right hand side of \eqref{equat} becomes
$$
\frac{-1}{n(n-1)(n-2)} + \frac{(-1)^{\frac{n}{2}}}{2(n-2)}- \frac{(-1)^{\frac{n}{2}}}{2n} - \sum_{\substack{N \textrm{ divides }(n-a),\\ a \in \{0,1,2\}, N >2}} (-1)^{\frac{n-a}{N}} \frac{\phi(N)}{\frac{(-1)^a+3}{2}(n-a)}, \textrm{ for } n \textrm{ even};
$$
$$
\frac{1}{n(n-1)(n-2)} - \sum_{\substack{N \textrm{ divides }(n-a),\\ a \in \{0,1,2\}, N >1}} (-1)^{\frac{n-a}{N}} \frac{\phi(N)}{\frac{(-1)^a+3}{2}(n-a)}, \textrm{ for } n \textrm{ odd.}
$$
A direct computation shows that both these expressions equal $1$: the left hand side of  \eqref{equat}. Here we used that
$$
\chi(\mathcal{M}_{0,n}/S_n)=1, \quad e(\mathcal{M}_{0,n})= (-1)^{n+1} (n-3)!;
$$
and the following elementary formula for the Euler totient function $\phi$:
$$
\sum_{d \textrm{ divides } n} (-1)^{\frac{n}{d}} \phi(d)= \begin{cases} 0 & n \textrm{ even}, \\ -n & n \textrm{ odd}.  \end{cases}
$$

\section{The inertia stack of moduli of smooth hyperelliptic curves}
\label{sezione3}
In this section we study the inertia stack of $\mathcal{H}_g$. We will implicitly use the fact that any family of hyperelliptic curves has a globally defined hyperelliptic involution, a result that follows from \cite[Theorem 5.5]{lonsted}. Let
$$
f\colon \mathcal{H}_g \to [\mathcal{M}_{0,2g+2}/S_{2g+2}]= \widetilde{\mathcal{M}}_{0,2g+2}
$$
be the map that associates to every hyperelliptic genus $g$ curve, the corresponding genus $0$ curve, together with the degree $2g+2$ \'etale Cartier divisor $D$ obtained by considering the branch locus of the hyperelliptic involution. This map is well defined on families as a consequence of \cite[Theorem 5.5]{lonsted}. 

Let $C\to C'=C/\langle \tau \rangle$ be a hyperelliptic curve, and $\alpha$ an automorphism of it. Then $\alpha$ induces an automorphism $\alpha^{red}$ of $C'$. If $D$ is the degree $2g+2$ branch divisor of $C \to C'$, then $\alpha^{red}$ induces a bijection on the set of reduced points of $D$. We can thus use the order of $\alpha^{red}$ ---instead of the order of $\alpha$---to reindex the components of the inertia stack:

\begin{definition} \label{reducedord} Let $I_N^{red}(\mathcal{H}_g)$ be the open and closed substack of $I(\mathcal{H}_g)$ whose objects correspond to pairs $(C, \alpha)$, where $C$ is an object of $\mathcal{H}_g$ and $\alpha^{red}\colon \mu_N \to \Aut (C/ \tau)$ is an injective homomorphism.
\end{definition} 

For our purposes, it is more convenient to work with $I_N^{red}(\mathcal{H}_g)$ than with the usual $I_N(\mathcal{H}_g)$. Note that of course we have in the end that
$$
I(\mathcal{H}_g)= \coprod_{N \in \mathbb{N}} I_N(\mathcal{H}_g)= \coprod_{N \in \mathbb{N}} I_N^{red}(\mathcal{H}_g)
$$
but with the latter decomposition, we have that the natural map of \ref{pullinertia}, $I'(f)\colon I(\mathcal{H}_g) \to f^*\left(I([\mathcal{M}_{0,2g+2}/S_{2g+2}])\right)$ induces maps:
\begin{displaymath}
I'(f)_N\colon I_N^{red}(\mathcal{H}_g) \to f^*\left(I_N([\mathcal{M}_{0,2g+2}/S_{2g+2}])\right) .
\end{displaymath}
This is not the case for the standard decomposition of the inertia stacks, since an automorphism of order $N$ of the genus $0$ curve can lift to an automorphism of order $N$ or to an automorphism of order $2N$ on the corresponding hyperelliptic curve.

Let $n=2g+2$ be the number of Weierstrass points, and $N=\ord( \alpha^{red})$. It is convenient to write (if possible) $n= kN+a$ for $k \in \mathbb{N}^+, a \in \{0,1,2\}$ (following the results of Section \ref{sezione2}). If $N>2$ such a decomposition of $n$ is unique. The number $a$ is the number of Weierstrass points whose image in the quotient via the hyperelliptic involution is a branch point for $\alpha^{red}$.

We label each twisted sector by a character. If $a$ is equal to zero, then there are four points in $C$ whose image in $C / \tau$ consists of the two points fixed by $\alpha^{red}$. In this case the automorphism $\alpha$ can: 
\begin{enumerate}
\item fix the four points;
\item exchange them two-by-two;
\item fix two of them and exchange the other two.
\end{enumerate}
In the first two cases, we label the twisted sectors by a pair $(\chi,1)$ or $(\chi,-1)$ respectively.

\begin{theorem} \label{inertiahg} For all $g \geq2$, we describe the twisted sectors of $\mathcal{H}_g$, the moduli stack of smooth hyperelliptic curves of genus $g$. Besides the untwisted sector, $I_1^{red}(\mathcal{H}_g)$ contains one more copy of $\mathcal{M}_{0,2g+2}/S_{2g+2}$, corresponding to the hyperelliptic involution. Let us now describe the cases when $N \geq 2$.
\begin{enumerate}
\item Suppose $N>2$. If there exists $a \in \{0,1,2\}$ such that $2 g+2=k N+a$: $$I_N^{red}(\mathcal{H}_{g})= \begin{cases} 
\coprod_{\chi \in \widetilde{\mu}_N^*, \lambda \in{\pm1}} \left(\mathcal{M}_{0,k+2}/S_k, (\chi,\lambda) \right)& a=0,  \ k \textrm{ even} \\
\coprod_{\chi \in \mu_N^*} \left(\mathcal{M}_{0,k+2}/S_k, \chi \right) & a=0, \ k \textrm{ odd}\\ 
\coprod_{\chi \in \mu_N^* \sqcup \mu_{2N}^*} \left(\mathcal{M}_{0,k+2}/S_k , \chi \right) & a=1 \\
\coprod_{\chi \in \widetilde{\mu}_{2N}^*} \left(\mathcal{M}_{0,k+2}/S_k, \chi \right) & a=2,  \ k \textrm{ even, } N \textrm{ even} \\
\coprod_{\chi \in  \widetilde{\mu}_{N}^*\sqcup \widetilde{\mu}_{2N}^*} \left(\mathcal{M}_{0,k+2}/S_k, \chi \right) & a=2,  \ k \textrm{ even, } N \textrm{ odd} \\
\coprod_{\chi \in \overline{\mu}_{2N}^*} \left(\mathcal{M}_{0,k+2}/S_k, \chi \right) & a=2,  \ k \textrm{ odd} 
\end{cases}$$
and $I_N^{red}(\mathcal{H}_g)$ is empty otherwise.
\item if $g$ is odd: $$I_2^{red}(\mathcal{H}_g) = \left(\mathcal{M}_{0, g+2}/S_g \times S_2, \zeta_4 \right) \coprod  \left(\mathcal{M}_{0, g+2}/S_g \times S_2,\zeta_4^3 \right)\coprod $$ $$\coprod \left(\mathcal{M}_{0, g+3}/S_{g+1} \times S_2,(-1,1) \right)\coprod \left(\mathcal{M}_{0, g+3}/S_{g+1}\times S_2,(-1,-1) \right) $$
where $\{\zeta_4, \zeta_4^3\} = \overline{\mu}_4^*= \mu_4^*$.
\item if $g$ is even: $$I_2^{red}(\mathcal{H}_g)= \left( \mathcal{M}_{0,g+2}/S_{g},-1 \right)\coprod \left( \mathcal{M}_{0,g+3}/S_{g+1},-1 \right) .$$

\end{enumerate}
\end{theorem}

\begin{proof}
First we observe that the morphism of \eqref{pullinertia}:
$$
I(f)_N\colon f^*\left(I_N([\mathcal{M}_{0,2g+2}/S_{2g+2}])\right) \to I_N([\mathcal{M}_{0,2g+2}/S_{2g+2}])
$$
is a $\mu_2$-gerbe, and as such it induces an isomorphism at the level of coarse moduli spaces.

Let us consider then
$$
I'(f)_N\colon I_N^{red}(\mathcal{H}_g) \to f^*\left(I_N([\mathcal{M}_{0,2g+2}/S_{2g+2}])\right).
$$
This map is a $2:1$ \'etale cover because every automorphism of a genus $0$ curve with an invariant smooth effective divisor of degree $2g+2$ can be lifted exactly to two automorphisms of the corresponding hyperelliptic curve. To prove the two points $(1)$ and $(2)$, we prove that this is the trivial cover, and then apply the result of Proposition \ref{inertian}. To prove point $(3)$, we show that in the particular case when $N= \ord (\alpha^{red})=2$ and $g$ is even, a lifting of $\alpha^{red}$ corresponds to a choice of a distinguished point $p$ in $D$, the branch divisor of $C \to C'$.

Let $C$ be a hyperelliptic curve, $\alpha$ an automorphism of it and $\tau$ the hyperelliptic involution. We have the two projections on the quotient:
$$\xymatrix{
C \ar[r]^{\hspace{-0.3cm}\pi} & C/\langle \tau \rangle \ar[r]^{\hspace{-0.3cm}p_N} & C/ \langle \tau, \alpha^{red} \rangle}.
$$
After choosing suitable coordinates on $C/\langle\tau \rangle \cong \mathbb{P}^1$ and $C/ \langle \tau, \alpha^{red} \rangle \cong \mathbb{P}^1$, the map $p_N$ is simply the map $z \to z^N$. Let $R$ be the set of ramification points of $p_N$. The number of points in $R$ that are branch points for $\pi$ is then $a$, by its very definition.

Now we study separately the three cases $a=0,1,2$.

\underline{If $a$ is equal to $0$}, then a hyperelliptic curve $C$ that admits an automorphism $\alpha$ of reduced order $N$ can be written as
$$
y^2=(x^N- \alpha_1) (x^N- \alpha_2) \ldots (x^N- \alpha_k)
$$
with the automorphism $\alpha$:
$$
\begin{cases}x \to \zeta_N^i x\\ y \to \pm y.
\end{cases}
$$
Exchanging the coordinates $0, \infty$, the action of $\alpha$ becomes:
$$
\begin{cases}x \to \zeta_N^{-i} x\\ y \to \pm (\zeta_N)^{i(g+1)} y.
\end{cases}
$$
If $k$ is odd, then $\alpha$ fixes two of the points in $\pi^{-1}(R)$ and exchanges the other two. The action of $\alpha$ on the two fixed fibers determines the same character in $\mu_N^*$. If $k$ is even, then $\alpha$ can either fix the four points of $\pi^{-1}(R)$ or exchange them two-by-two. In both the cases, the action of $\alpha$ or $\alpha \tau$ on the four fixed points determines an element of $\widetilde{\mu}_N^*$.

\underline{If $a$ is equal to $1$}, then a hyperelliptic curve $C$ that admits an automorphism $\alpha$ of reduced order $N$ can be written as
$$
y^2=x (x^N- \alpha_1) (x^N- \alpha_2) \ldots (x^N- \alpha_k)
$$
with the automorphism $\alpha$:
$$
\begin{cases}x \to \zeta_N^i x\\ y \to \pm \zeta_{2N}^i y.
\end{cases}
$$
In this case, if we call $p$ the point in $R$ that is also a branch point of $\pi$, then the action of $\alpha$ on $T_{\pi^{-1}(p)}$ determines a well-defined element of $\mu_N^*$ or $\mu_{2N}^*$.

\underline{If $a$ is equal to $2$}, then a hyperelliptic curve $C$ that admits an automorphism $\alpha$ of reduced order $N$ can be written as
$$
y^2=x (x^N- \alpha_1) (x^N- \alpha_2) \ldots (x^N- \alpha_k)
$$
with the automorphism $\alpha$:
$$
\begin{cases}x \to \zeta_N^i x\\ y \to \pm \zeta_{2N}^i y.
\end{cases}
$$
Exchanging the coordinates $0, \infty$, the action of $\alpha$ becomes
$$
\begin{cases}x \to \zeta_N^{-i} x\\ y \to \pm \zeta_{2N}^{-i} (\zeta_N)^{i g} y.
\end{cases}
$$
In this case, the action of $\alpha$ fixes the two points in $\pi^{-1}(R)$. Then $\alpha$ induces a well-defined element of $\widetilde{\mu}_{2N}^*$ when $k$ is even and $N$ is even, of $\widetilde{\mu}_N^* \sqcup \widetilde{\mu}_{2N}^*$ when $k$ is even and $N$ is odd, and of $\overline{\mu}_{2N}^*$ when $k$ is odd (and therefore $N$ is even).

Now for the point $(2)$, it is enough to check that our separate study in the different cases $a=0,1,2$ carries on also when $N= \ord(\alpha^{red})=2$, if $g$ is odd.

The two remaining cases are when $g$ is even, $N=2$; therefore $a$ is equal to zero (then $k=g+1$ is odd), or $a$ is equal to two (then $k=g$ is also even). We have that $\mu_2^*= \widetilde{\mu}_4^*= \widetilde{\mu}_2^*$. In these cases, the action of $\alpha$ on $\pi^{-1}(R)$ distinguishes the two points of $R$.  For example, if $a=2$, $k$ even, then $\alpha$ acts on the two points of $\pi^{-1}(R)$, on one of them with character $\zeta_4$ and on the other with character $\zeta_4^3$. In these cases therefore, the two $2:1$ \'etale covers, at the level of coarse moduli spaces, are the two quotient maps:
$$
\mathcal{M}_{0,g+2}/S_g \to \mathcal{M}_{0,g+2}/S_g \times S_2 \quad \textrm{and} \quad \mathcal{M}_{0,g+3}/S_{g+1} \to \mathcal{M}_{0,g+3}/S_{g+1} \times S_2.
$$
\end{proof}

\begin{remark} The earlier-mentioned theorem could be restated as a stack-theoretic description of the rigidified inertia stack of $\mathcal{H}_g$ (cfr. Remark \ref{stackyremark}), further rigidified along the hyperelliptic involution. Each (doubly rigidified, stacky) twisted sector is the \emph{stack} quotient $[\mathcal{M}_{0,k+2}/S]$ in place of $\mathcal{M}_{0,k+2}/S$.
\end{remark}

We observe that the consistency check we used to verify the correctness of Proposition \ref{inertian}, see \eqref{equat} and comments thereof, does not produce any further check on Theorem \ref{inertiahg}. Indeed, the most delicate point of Theorem \ref{inertiahg} is understanding when the double covers $I'(f)_N$ are trivial and when they are not, a distinction that does not affect the orbifold Euler characteristic.

\section{The orbifold cohomology of moduli of smooth hyperelliptic curves}
\label{sezione4}
Here we compute the orbifold Poincar\'e polynomial for moduli of smooth hyperelliptic curves (see Definition \ref{defcoomorb2}). 

Let us fix a hyperelliptic curve $C$ of genus $g$. A basis for the cotangent space $(T_C \mathcal{H}_g)^{\vee}$ is given by
$$
\left(\frac{d X}{Y}\right)^2 \quad X\left(\frac{d X}{Y}\right)^2 \quad \ldots \quad X^{2g-2} \left(\frac{d X}{Y}\right)^2
$$
If $\alpha$ is an automorphism of $C$, it is straightforward to compute its action on each element of such a basis.
What we have done so far gives us the possibility of writing a closed formula for $P_{\mathcal{H}_g}^{CR}$ for fixed $g\geq 2$.

\begin{theorem} \label{principale} The orbifold Poincar\'e polynomial of moduli of smooth hyperelliptic curves is given by the formula
\begin{displaymath}
P^{CR}_{\mathcal{H}_g}(q)=\sum_{(k,N,i) \in A_{2g+2}} q^{a_g(i,N)} P^0_{k+2;k,1,1}(q)+  \sum_{(k,N,i) \in A_{2g+1}} 2 q^{b_g(i,N)} P^0_{k+2;k,1,1}(q)+  
\end{displaymath}
\begin{displaymath}
+\sum_{(k,N,i) \in A_{2g}} q^{b_g(i,N)}  P^0_{k+2;k,1,1}(q) +2+\begin{cases} q^{\frac{g-1}{2}} P^0_{g+3;g+1,1,1}(q)+ q^{\frac{g}{2}} P^0_{g+2;g,1,1}(q)& \textrm{if } g \textrm{ is even}\\2q^{\frac{g-1}{2}} P^0_{g+3;g+1,2}(q)+2 q^{\frac{g}{2}} P^0_{g+2;g,2}(q) & \textrm{if } g \textrm{ is odd}, 
\end{cases} 
\end{displaymath}
where the sets of indices are defined as $$A_{n}:=\left\{(k,N,i) \in \mathbb{N}^2 \times \mathbb{Z}_N^*| \ N>2, \ k N= n \right\}$$
and the exponents are
$$
a_g(i,N):= 2\left( 2g-1- \sum_{j=1}^{2g-1} \left\{ \frac{i (j+1)}{N}\right\} \right)
$$
$$
b_g(i,N):= 2 \left( 2g-1 -\sum_{j=1}^{2g-1} \left\{ \frac{i j}{N}\right\} \right).
$$

\end{theorem}

\begin{remark} By substituting all the Poincar\'e polynomials on the right-hand side of Theorem \ref{principale} with $1$, one gets the \virg{polynomial} whose coefficients in degree $i \in \mathbb{Q}^{\geq 0}$ are the dimensions of the stringy Chow group of degree $i$ (cf. Remark \ref{stringychow}). This is so because all the twisted sectors of $\mathcal{H}_g$ have trivial Chow group, since their coarse moduli spaces are quotients of affine sets.
\end{remark}
In particular, we can write closed formulas for the total dimensions of the orbifold cohomology of $\mathcal{H}_g$. Let us define
$$h_{CR}(g):= \dim H^*_{CR}(\mathcal{H}_g).$$
We denote with $\phi$ the Euler totient function. Then we can give a corollary of 
Theorem \ref{inertiahg}:
\begin{corollary} The following explicit formulas for the function just introduced hold:
\begin{enumerate}
\item If $g$ is even, $n=2g+2$:
$$
h_{CR}(g)= 3+ 2g +  2 \sum_{\substack{ n=kN+a, \\N>2, a \in \{0,1,2 \}}} k \phi(N).
$$
\item If $g$ is odd, $n=2g+2$:
$$
h_{CR}(g)= 2+ 4\left( \lfloor \frac{n-2}{4} \rfloor +\lfloor \frac{n-1}{4} \rfloor \right) +2 \sum_{\substack{ n=kN+a, \\N>2, a \in \{0,1,2 \}}} k \phi(N).
$$
\end{enumerate}
\end{corollary}
\medskip

\section{Future directions}

A series of natural questions arise as a consequence of this work. 
\begin{enumerate}
\item One can address the question of studying the orbifold cohomology of the Deligne--Mumford compactification $\overline{\mathcal{H}}_g$. Using the theory of admissible double covers, it is straightforward to extend Proposition \ref{inertian} and Theorem \ref{inertiahg} to describe (respectively) the coarse moduli spaces of the closures of the inertia stacks of $[{\mathcal{M}}_{0,n}/S_n]$ and ${\mathcal{H}}_g$ inside the inertia stacks of $[\overline{\mathcal{M}}_{0,n}/S_n]$ and $\overline{\mathcal{H}}_g$ (a similar study is performed for $\mathcal{M}_3$ in \cite[Section 4]{paganitommasi}). It is an interesting combinatorial problem to describe and study all the remaining twisted sectors of $[\overline{\mathcal{M}}_{0,n}/S_n]$ and $\overline{\mathcal{H}}_g$.
\item Can one give a reasonably understandable description of the ring structure, thus computing the Chen--Ruan cohomology ring of $\mathcal{H}_g$ and $\overline{\mathcal{H}}_g$? This would require a good description of the so-called second inertia stack, or of the moduli stack of $3$-pointed stable maps\footnote{See \cite{abvis2} for the moduli space of stable maps to a Deligne--Mumford stack.}:
$$
\overline{\mathcal{M}}_{0,3}(\mathcal{H}_g, 0), \quad \overline{\mathcal{M}}_{0,3}(\overline{\mathcal{H}}_g, 0),
$$
in the same spirit of Theorem \ref{inertiahg}. See \cite{pagani1} and \cite{pagani2} for a description of the Chen--Ruan cohomology ring of $\overline{\mathcal{M}}_{1,n}$ and $\overline{\mathcal{M}}_{2,n}$.
\item Another interesting geometric question is to study the inertia stack for the moduli stacks of $k$-gonal curves, $k>2$. In the case of the moduli stacks of \emph{cyclic} $k$-gonal covers, one can adopt the same strategy of the present paper and solve the problem in two steps: studying the twisted sectors of the quotient stack of $\mathcal{M}_{0,n}$ by a certain subgroup of $S_n$ (in analogy with \ref{inertian}), and studying a $\mu_k$-torsor over each such twisted sector (in analogy with \ref{inertiahg}).
\end{enumerate}
\ \\
\begin{center}
\textsc{Acknowledgments} 
\end{center}
The author is grateful to Gilberto Bini, Torsten Ekedahl, Carel Faber and Barbara Fantechi, for useful discussions and help. The author is grateful to the anonymous referees for useful suggestions and remarks.
\begin{center}
\textsc{Funding} 
\end{center}
This project was supported by the Wallenberg foundation, grant KAW 2005.0098, and took place at KTH Royal Institute of Technology. It was also partly supported  by   {\sc prin}    ``Geometria delle variet\`a algebriche e dei loro spazi di moduli'', by Istituto Nazionale di Alta Matematica.




\end{document}